\documentclass[a4paper,11pt]{amsart}
\usepackage{wrapfig}
\usepackage[dvips]{graphicx}
\usepackage{amsmath,amsthm,amssymb,amscd}
\theoremstyle{definition}
\newtheorem{thm}{Theorem}[section]
\newtheorem{Def}[thm]{Definition}
\newtheorem{pro}[thm]{Proposition}
\newtheorem{cor}[thm]{Corollary}
\newtheorem{lem}[thm]{Lemma}
\newtheorem{ex}[thm]{Example}
\newtheorem{rem}[thm]{Remark}
\theoremstyle{definition}

\begin{document}

\title{Fundamental group of simple $C^*$-algebras with unique trace III}
\author{Norio Nawata}
\address[Norio Nawata]{Graduate School of Mathematics, 
Kyushu University, Motooka, 
Fukuoka, 819-0395,  Japan}      
\email{n-nawata@math.kyushu-u.ac.jp}
\maketitle
\begin{abstract}
We introduce the fundamental group ${\mathcal F}(A)$ of 
a simple $\sigma$-unital $C^*$-algebra $A$ with unique (up to scalar multiple) 
densely defined lower semicontinuous trace. 
This is a generalization of \cite{NW} and \cite{NW2}. 
Our definition in this paper makes sense for stably projectionless $C^*$-algebras. 
We show that there exist separable stably projectionless $C^*$-algebras such that 
their fundamental groups are equal to $\mathbb{R}_+^\times$ 
by using the classification theorem of Razak \cite{Raz} and Tsang \cite{Tsa1}. 
This is a contrast to the unital case in \cite{NW} and \cite{NW2}. 
This study is motivated by the work of Kishimoto and Kumjian in \cite{KK1}. 
\end{abstract}

\section{Introduction}
Let $M$ be a factor of type $\mathrm{II}_1$ with a normalized trace 
$\tau$. Murray and von Neumann  introduced 
the fundamental group ${\mathcal F}(M)$ of $M$ in \cite{MN}. 
They showed that if $M$ is  hyperfinite, then 
${\mathcal F}(M) = {\mathbb R_+^{\times}}$. 
Since then 
there has been many works on the computation of the 
fundamental groups. Voiculescu \cite{Vo} showed that 
${\mathcal F}(L(\mathbb{F}_{\infty}))$ of the group factor 
of the free group $\mathbb{F}_{\infty}$ contains the positive rationals and 
Radulescu proved that 
${\mathcal F}(L(\mathbb{F}_{\infty})) = {\mathbb R}_+^{\times}$ in 
\cite{Ra}.  Connes \cite{Co} showed that if $G$ is an ICC group with property 
(T), then  ${\mathcal F}(L(G))$ is a countable group. Popa 
showed that any countable subgroup of $\mathbb R_+^{\times}$ 
can be realized as the fundamental group of some 
factor of type $\mathrm{II}_1$ in \cite{Po1}. 
Furthermore Popa and Vaes \cite{PV} exhibited a large family $\mathcal{S}$ 
of subgroups of $\mathbb{R}_{+}^\times$, containing $\mathbb{R}_{+}^\times$ 
itself, all of its countable subgroups, as well as uncountable subgroups with 
any Hausdorff dimension in $(0,1)$, such that for each $G\in\mathcal{S}$ 
there exist many free ergodic measure preserving actions of $\mathbb{F}_{\infty}$ 
for which the associated $\mathrm{II}_1$ factor $M$ has the fundamental group equal to $G$. 
In our previous paper \cite{NW}, 
we introduced the fundamental group $\mathcal{F}(A)$ 
of a simple unital $C^*$-algebra $A$ with a normalized trace $\tau$ 
based on the computation  of Picard groups by 
Kodaka \cite{kod1}, \cite{kod2} and \cite{kod3}. 
The fundamental group ${\mathcal F}(A)$ is 
defined as the set of 
the numbers $\tau \otimes Tr(p)$ for some projection 
$p \in M_n(A)$ such that $pM_n(A)p$ is isomorphic to $A$. 
We computed the fundamental groups of several $C^*$-algebras and showed 
that any countable subgroup of $\mathbb{R}_+^\times$ 
can be realized as the fundamental group of a separable simple unital $C^*$-algebra 
with a unique trace in \cite{NW} and \cite{NW2}. 
Note that the fundamental groups of separable simple unital $C^*$-algebras are countable. 

In this paper 
we introduce the fundamental group of a simple $\sigma$-unital $C^*$-algebra 
with unique (up to scalar multiple) densely defined 
lower semicontinuous trace. 
We do not assume that $C^*$-algebras are unital. 
In particular, our definition in this paper makes sense for 
stably projectionless $C^*$-algebras. 
Let $A$ be a $\sigma$-unital simple $C^*$-algebra with unique (up to scalar multiple) densely defined 
lower semicontinuous trace $\tau$. 
The fundamental group ${\mathcal F}(A)$ of $A$ is defined as the set of 
the numbers $d_\tau (h_1)/d_\tau (h_2)$ for some nonzero positive elements 
$h_1,h_2 \in A\otimes \mathbb{K}$ such that 
$\overline{h_1(A\otimes \mathbb{K})h_1}$ is isomorphic to 
$\overline{h_2(A\otimes \mathbb{K})h_2}$ and $d_\tau (h_2)<\infty$ 
where $d_\tau$ is the dimension function defined by $\tau$. 
Then the fundamental group ${\mathcal F}(A)$ of $A$ is a 
multiplicative subgroup of ${{\mathbb R}_+^{\times}}$. 
We show that if $A$ is unital, then our definition in this paper coincides 
with previous definition in \cite{NW} and \cite{NW2}. 
Hence if $A\otimes\mathbb{K}$ is separable and has a nonzero projection, then ${\mathcal F}(A)$ 
is a countable multiplicative subgroup of ${{\mathbb R}_+^{\times}}$. 
By contrast, we show that there exist separable simple stably projectionless 
$C^*$-algebras such that their fundamental groups are equal to $\mathbb{R}_+^\times$ 
by using the classification theorem of Razak \cite{Raz} and Tsang \cite{Tsa1}.  
This study is motivated by the work of Kishimoto and Kumjian in \cite{KK1}. 
(See Example \ref{ex:KK}.) 

\section{Hilbert $C^*$-modules and Induced traces}\label{sec:pre} 
We say a $C^*$-algebra $A$ is $\sigma$-{\it unital} if 
$A$ has a countable approximate unit. 
In particular, if $A$ is $\sigma$-unital, then there exists a positive element 
$h\in A$ such that $\{h^{\frac{1}{n}}\}_{n\in\mathbb{N}}$ is an approximate 
unit. Such a positive element $h$ is called {\it strict positive} in $A$. 
Let $\mathcal{X}$ be a right Hilbert $A$-module 
and let $\mathcal{H}(A)$ denote the 
set of isomorphic classes $[\mathcal{X}]$ of 
right Hilbert $A$-modules. 
(See \cite{Lan} and \cite{MT} for the basic facts on Hilbert modules.) 
We denote by $L_A(\mathcal{X})$ 
the algebra of the adjointable operators on $\mathcal{X}$. 
For $\xi,\eta \in \mathcal{X}$, a  "rank one operator" $\Theta_{\xi,\eta}$ 
is defined by $\Theta_{\xi,\eta}(\zeta) 
= \xi \langle\eta,\zeta\rangle_A$ for $\zeta \in \mathcal{X}$. 
We denote by $K_A(\mathcal{X})$ the closure 
of the linear span of "rank one operators" $\Theta_{\xi,\eta}$ 
and by $\mathbb{K}$ the $C^*$-algebra of compact operators on a separable infinite 
Hilbert space. 
Let $\mathcal{X}_A$ be a right Hilbert $A$-module $A$ with the obvious right $A$-action and 
$\langle a ,b\rangle_A = a^*b$ for $a,b \in A$. 
Then $K_A(\mathcal{X}_A)$ is isomorphic to $A$. 
Hence if $A$ is unital, then $K_A(\mathcal{X}_A)=L_A(\mathcal{X}_A)$. 
The multiplier algebra, denote by $M(A)$, of a $C^*$-algebra $A$ is 
the largest unital $C^*$-algebra that contains $A$ as an essential ideal. 
It is unique up to isomorphism over $A$ and isomorphic to $L_A(\mathcal{X}_A)$. 
Let $H_A$ denote the standard Hilbert module 
$\{(x_n)_{n\in \Bbb{N}};x_n\in A,\sum x_n^*x_n\;\mathrm{converges}\; \mathrm{in}\; A\}$ 
with an $A$-valued inner product 
$\langle (x_n)_{n\in\Bbb{N}},(y_n)_{n\in\Bbb{N}}\rangle =\sum x_n^*y_n$. 
Then there exists a natural isomorphism $\psi$ of $A\otimes\mathbb{K}$ to 
$K_A(H_A)$ and $\psi$ can be uniquely extended to an isomorphism $\tilde{\psi}$ of 
$M(A\otimes\mathbb{K})$ to $L_A(H_A)$. 
For simplicity of notation, 
we use the same letter $x$ for $\tilde{\psi} (x)$ where $x\in M(A\otimes\mathbb{K})$. 

A finite subset $\{\xi_i\}_{i=1}^n$ of $\mathcal{X}$ is called a {\it finite basis} if 
$\eta =\sum_{i=1}^n\xi_i\langle\xi_i,\eta\rangle_A$ for any $\eta\in\mathcal{X}$. 
More generally, we call a sequence $\{\xi_i\}_{i\in \mathbb{N}}\subseteq \mathcal{X}$ a 
{\it countable basis} of $\mathcal{X}$ if 
$\eta =\sum_{i=1}^\infty\xi_i\langle\xi_i,\eta\rangle_A$ in norm 
for any $\eta\in\mathcal{X}$, see \cite{KPW}, \cite{KW} and \cite{W}. 
It is also called a standard normalized tight frame as in \cite{FL} and \cite{FL2}. 
A countable basis $\{\xi_i\}_{i\in \mathbb{N}}$ always converges unconditionally, that is, 
for any $\eta\in\mathcal{X}$, the net associating $\sum_{i\in F}\xi_i\langle\xi_i,\eta\rangle_A$ 
to each finite subset $F\subseteq\mathbb{N}$ is norm converging to $\eta$. 
It is a consequence of the following estimate: for every $\xi\in\mathcal{X}$, 
$a,b\in K_A(\mathcal{X})$, with $0\leq a\leq b\leq 1$, 
$\| \xi -b\xi\|^2\leq \|\xi\|\|\xi -a\xi\|$. 
\begin{pro}\label{pro:basis}
Let $A$ be a simple $C^*$-algebra and $\mathcal{X}$ a right Hilbert $A$-module. 
Assume that $K_A(\mathcal{X})$ is $\sigma$-unital. 
Then $\mathcal{X}$ has a countable basis. 
\end{pro}
\begin{proof}
Consider a right ideal $\{\Theta_{\xi_0,\zeta}:\zeta\in\mathcal{X}\}$ 
in $K_A(\mathcal{X})$ for some $\xi_0\in\mathcal{X}$. 
Then a similar argument in \cite{B} (Lemma 2.3) proves the proposition. 
\end{proof} 

\begin{rem}
In general, we need not assume that $A$ is simple. 
If $K_A(\mathcal{X})$ is $\sigma$-unital, then $\mathcal{X}$ has a countable basis. 
This is an immediate consequence of Kasparov's stabilization trick \cite{Kas}. 
\end{rem}

Let $B$ be a $C^*$algebra. 
An $A$-$B$-{\it equivalence bimodule} is an $A$-$B$-bimodule $\mathcal{F}$ which is simultaneously a 
full left Hilbert $A$-module under a left $A$-valued inner product $_A\langle\cdot ,\cdot\rangle$ 
and a full right Hilbert $B$-module under a right $B$-valued inner product $\langle\cdot ,\cdot\rangle_B$, 
satisfying $_A\langle\xi ,\eta\rangle\zeta =\xi\langle\eta ,\zeta\rangle_B$ for any 
$\xi, \eta, \zeta \in \mathcal{F}$. We say that $A$ is {\it Morita equivalent} to $B$ 
if there exists an $A$-$B$-equivalence bimodule. 
It is easy to see that $K_B(\mathcal{F})$ is isomorphic to $A$. 
A dual module $\mathcal{F}^*$ of an $A$-$B$-equivalence bimodule $\mathcal{F}$ is a set 
$\{\xi^* ;\xi\in\mathcal{F} \}$ with the operations such that $\xi^* +\eta^*=(\xi +\eta )^*$, 
$\lambda\xi ^*=(\overline{\lambda}\xi)^*$, $b\xi^* a=(a^*\xi b^*)^*$, 
$_B\langle\xi^*,\eta^*\rangle =\langle\eta ,\xi\rangle_B$ and 
$\langle \xi^*,\eta^*\rangle_A =\;_A\langle\eta ,\xi\rangle$. 
The bimodule $\mathcal{F}^*$ is a $B$-$A$-equivalence bimodule. 
We refer the reader to \cite{RW} and \cite{R2} for the basic facts on 
equivalence bimodules and Morita equivalence. 

We review basic facts on the Picard groups of 
$C^*$-algebras introduced by Brown, Green and Rieffel 
in \cite{BGR}. 
For $A$-$A$-equivalence bimodules 
$\mathcal{E}_1$ and 
$\mathcal{E}_2$, we say that $\mathcal{E}_1$ is isomorphic to $\mathcal{E}_2$ as an equivalence 
bimodule if there exists a $\mathbb{C}$-liner one-to-one map $\Phi$ of $\mathcal{E}_1$ onto 
$\mathcal{E}_2$ with the properties such that $\Phi (a\xi b)=a\Phi (\xi )b$, 
$_A\langle \Phi (\xi ) ,\Phi(\eta )\rangle =\;_A\langle \xi ,\eta\rangle$ and 
$\langle \Phi (\xi ) ,\Phi(\eta )\rangle_A =\langle\xi,\eta\rangle_A$ for $a,b\in A$, 
$\xi ,\eta\in\mathcal{E}_1$. 
The set of isomorphic classes $[\mathcal{E}]$ of the $A$-$A$-equivalence 
bimodules $\mathcal{E}$ forms a group under the product defined by 
$[\mathcal{E}_1][\mathcal{E}_2] = [\mathcal{E}_1 \otimes_A \mathcal{E}_2]$. 
We call it the {\it Picard group} of $A$ and denote it  by  $\mathrm{Pic}(A)$. 
The identity of $\mathrm{Pic}(A)$ is given by 
the $A$-$A$-bimodule $\mathcal{E}:= A$ with  
$\; _A\langle a_1 ,a_2 \rangle = a_1a_2^*$ and $\langle a_1 ,a_2\rangle_A = a_1^*a_2$ for 
$a_1,a_2 \in A$. The inverse element of $[\mathcal{E}]$ in the Picard group of $A$ 
is the dual module $[\mathcal{E}^*]$. 
Let $\alpha$ be an automorphism of $A$, and let 
$\mathcal{E}_{\alpha}^A=A$ with the obvious left $A$-action and the obvious $A$-valued inner product. 
We define the right $A$-action on $\mathcal{E}_\alpha^A$ by 
$\xi\cdot a=\xi\alpha(a)$ for 
any $\xi\in\mathcal{E}_\alpha^A$ and $a\in A$, and the right $A$-valued inner product by 
$\langle\xi ,\eta\rangle_A=\alpha^{-1} (\xi^*\eta)$ for any $\xi ,\eta\in\mathcal{E}_\alpha^A$.
Then $\mathcal{E}_{\alpha}^A$ is an $A$-$A$-equivalence bimodule. For $\alpha, \beta\in\mathrm{Aut}(A)$, 
$\mathcal{E}_\alpha^A$ is isomorphic to $\mathcal{E}_\beta^A$ if and only if 
there exists a unitary $u \in A$ such that 
$\alpha = ad \ u \circ \beta $. Moreover, ${\mathcal E}_\alpha^A \otimes 
{\mathcal E}_\beta^A$ is 
isomorphic to $\mathcal{E}_{\alpha\circ\beta}^A$. Hence we obtain an homomorphism $\rho_A$ 
of $\mathrm{Out}(A)$ to $\mathrm{Pic}(A)$. 
An $A$-$B$-equivalence bimodule $\mathcal{F}$ induces an isomorphism $\Psi$ 
of $\mathrm{Pic}(A)$ to $\mathrm{Pic}(B)$ by 
$\Psi ([\mathcal{E}])=[\mathcal{F}^*\otimes\mathcal{E}\otimes\mathcal{F}]$ 
for $[\mathcal{E}]\in\mathrm{Pic}(A)$. 
Therefore if $A$ is Morita equivalent to $B$, then $\mathrm{Pic}(A)$ is isomorphic to 
$\mathrm{Pic}(B)$. 

If $A$ is unital, then any 
$A$-$B$-equivalence bimodule $\mathcal{F}$ is a finitely generated projective $B$-module as a right 
module with a finite basis $\{\xi_i\}_{i=1}^n$. 
Put $p=(\langle\xi_i,\xi_j\rangle_A)_{ij} \in M_n(B)$. 
Then $p$ is a projection and $\mathcal{F}$ is isomorphic to 
$pB^n$ as a right Hilbert $B$-module 
with an isomorphism  of $A$ to  $pM_n(B)p$. 
In the case $A$ is $\sigma$-unital, 
an $A$-$B$-equivalence bimodule $\mathcal{F}$ has a countable basis 
$\{\xi_i\}_{i\in\mathbb{N}}$ as a right Hilbert $B$-module by Proposition \ref{pro:basis}. 
Define $p$ by $p(b_n)_n=(\sum_{m=1}^\infty \langle\xi_n,\xi_m\rangle_Bb_m)_n$ for 
$(b_n)_n\in H_B$. Then $p$ is a projection in $L_B(H_B)$ by \cite{KPW2} (Lemma 2.1) and 
$\sum_{i=1}^n\Theta_{\xi_i,\xi_i}\leq 1_{L_B(H_B)}$ for any $n\in\mathbb{N}$. 
Therefore $\mathcal{F}$ is isomorphic to 
$pH_B$ as a right Hilbert module with an isomorphism of $A$ to 
$p(B\otimes\mathbb{K})p$. 
\begin{pro}\label{pro:module}
Let $A$ and $B$ be $\sigma$-unital simple $C^*$-algebras and $\mathcal{F}$ 
an $A$-$B$-equivalence bimodule. Then there exists a positive element 
$h\in A\otimes\mathbb{K}$ such that $\mathcal{F}$ is isomorphic to $\overline{hH_B}$ as 
a right Hilbert $B$-module with an isomorphism of $A$ to $\overline{h(B\otimes\mathbb{K})h}$. 
\end{pro}
\begin{proof}
By the discussion above, 
there exists a projection $p\in M(B\otimes\mathbb{K})$ such that 
an $A$-$B$-equivalence bimodule $\mathcal{F}$ is isomorphic to $pH_B$ as a right Hilbert 
$B$-module with an isomorphism of $A$ to  $p(B\otimes\mathbb{K})p$. 
Since $p(B\otimes\mathbb{K})p$ is $\sigma$-unital, there exists a strict positive element 
$h$ in $p(B\otimes\mathbb{K})p$. It is easy to see that 
$pH_B=\overline{hH_B}$. Therefore $\mathcal{F}$ is isomorphic to $\overline{hH_B}$ as 
a right Hilbert $B$-module with an isomorphism of $A$ to $\overline{h(B\otimes\mathbb{K})h}$ 
as a $C^*$-algebra. 
\end{proof}

Recall that a {\it trace} on $A$ is a linear map $\tau$ on the positive elements of $A$, 
with values in $[0,\infty ]$ that vanishes at $0$ and satisfies 
the trace identity $\tau (a^*a)=\tau (aa^*)$. 
If $A$ is simple, then $\tau (a^*a)=0$ implies $a=0$. 
Define $\mathcal{M}_\tau^+=
\{x\geq 0: \hat{\tau}  (x)<\infty \}$ and 
$\mathcal{M}_\tau =\mathrm{span} \mathcal{M}_\tau^+$. 
Then $\mathcal{M}_\tau $ is an ideal in $A$. 
Every trace $\tau$ on $A$ extends a positive linear map on $\mathcal{M}_\tau $. 
A {\it normalized trace} is a state on $A$ which is a trace. 
We say $\tau$ is {\it densely defined} if $\mathcal{M}_\tau $ is a dense ideal in $A$. 
In particular, each densely defined trace on $A$ extends a positive linear map 
on the Pedersen ideal $Ped(A)$, which is the minimal dense ideal in $A$. (See \cite{Ped2}.) 
Note that if $A$ is unital, then every densely defined trace is bounded. 
We review some results about inducing traces from a simple $\sigma$-unital $C^*$-algebra $A$ 
through a right Hilbert $A$-module $\mathcal{X}$. 
See, for example, \cite{CZ}, \cite{CP},\cite{IKW},\cite{LN},\cite{Ped} and \cite{PWY} 
for induced traces in several settings. 
We state the relevant properties in a way that is convenient our purposes, and we 
include a self-contained proof. 
\begin{pro}\label{pro:ind}
Let $A$ and $\mathcal{X}$ be as above and let $\tau$ be a densely defined lower 
semicontinuous trace. 
For $x\in K_A(\mathcal{X})_+$ (resp. $L_A(\mathcal{X})_+$ ), 
define 
\[Tr_\tau^\mathcal{X}(x):=\sum_{i=1}^\infty \tau (\langle \xi_i,x\xi_i\rangle_A) \]
where $\{\xi_i\}_{i=1}^\infty$ is a countable basis of $\mathcal{X}$. 
Then $Tr_\tau^\mathcal{X}$ does not depend on the choice of basis and is a densely defined (resp. strictly 
densely defined) lower semicontinuous trace on $K_A(\mathcal{X})$ (resp. $L_A(\mathcal{X})$). 
\end{pro}
\begin{proof}
Let $\{\xi_i\}_{i\in\mathbb{N}}$ and 
$\{\zeta_k\}_{k\in\mathbb{N}}$ be countable bases of  $\mathcal{X}$. 
For any positive element $x\in K_A(\mathcal{X})_{+}$, 
\begin{align*}
\lim_{n\rightarrow \infty}\sum_{i=1}^n\tau (\langle\xi_i, x\xi_i\rangle_A) 
& =\lim_{n\rightarrow\infty}\sum_{i=1}^n\tau (\langle\sum_{k=1}^\infty\zeta_k 
\langle\zeta_k, x^{\frac{1}{2}}\xi_i\rangle_A, x^{\frac{1}{2}}\xi_i\rangle_A) \\
& =\lim_{n\rightarrow\infty}\sum_{i=1}^n\tau (\sum_{k=1}^\infty
\langle x^{\frac{1}{2}}\xi_i, \zeta_k\rangle_A\langle \zeta_k, x^{\frac{1}{2}}\xi_i\rangle_A).
\end{align*} 
By the lower semicontinuity of $\tau$ and $\sum_{i=1}^n\Theta_{\xi_i,\xi_i}\leq 1_{L_A(\mathcal{X})}$,  
\begin{align*}
\lim_{n\rightarrow \infty}\sum_{i=1}^n\tau (\sum_{k=1}^\infty 
\langle x^{\frac{1}{2}}\xi_i, \zeta_k\rangle_A\langle \zeta_k, x^{\frac{1}{2}}\xi_i\rangle_
A) 
& =\lim_{n\rightarrow\infty}\sum_{i=1}^n\sum_{k=1}^\infty\tau (
\langle x^{\frac{1}{2}}\xi_i, \zeta_k\rangle_A\langle \zeta_k, x^{\frac{1}{2}}\xi_i\rangle_A) \\
& =\lim_{n\rightarrow\infty}\sum_{i=1}^n\sum_{k=1}^\infty\tau (
\langle x^{\frac{1}{2}}\zeta_k, \xi_i\rangle_A\langle \xi_i, x^{\frac{1}{2}}\zeta_k\rangle_A) \\
& =\lim_{n\rightarrow\infty}\sum_{k=1}^\infty\tau (\langle x^{\frac{1}{2}}\zeta_k, 
\sum_{i=1}^n\xi_i\langle\xi_i, x^{\frac{1}{2}}\zeta_k\rangle_A\rangle_A) \\
& \leq \lim_{m\rightarrow \infty}\sum_{k=1}^m\tau (\langle\zeta_k, x\zeta_k\rangle_A).
\end{align*}
Therefore $Tr_\tau^\mathcal{X}$ does not depend on the choice of basis. A similar argument 
implies that $Tr_\tau^\mathcal{X}(x^*x)=Tr_\tau^\mathcal{X}(xx^*)$ for $x\in L_A(\mathcal{X})$. 

We shall show that $Tr_\tau^\mathcal{X}$ is densely defined on $K_A(\mathcal{X})$. 
Since $K_A(\mathcal{X})$ is simple, it is enough to show that there exists a nonzero 
element $x\in K_A(\mathcal{X})$ such that $Tr_\tau^\mathcal{X}(x)<\infty$. 
There exists a nonzero positive element $a\in A$ such that $\tau (a)<\infty$ 
because $\tau$ is densely defined on $A$. 
For any $\eta\in\mathcal{X}$, we have 
$\langle\eta a^{\frac{1}{2}},\eta a^{\frac{1}{2}}\rangle_A
\leq\|\langle\eta,\eta\rangle_A\| a$, so 
$\tau (\langle\eta a^{\frac{1}{2}},\eta a^{\frac{1}{2}}\rangle_A) )<\infty$. 
By the simplicity of $A$, there exists an element $\eta_0$ in $\mathcal{X}$ such that 
$\eta_0a^{\frac{1}{2}}\neq 0$. Define $\zeta :=\eta_0a^{\frac{1}{2}}$. 
Then we have 
\begin{align*}
Tr_\tau^\mathcal{X}(\Theta_{\zeta ,\zeta})
& =\sum_{n=1}^\infty\tau (\langle\xi_n,\zeta\rangle_A\langle\zeta,\xi_n\rangle_A) \\
& =\lim_{N\rightarrow\infty}\tau (\langle\zeta, \sum_{n=1}^N\xi_n\langle\xi_n,\zeta\rangle_A
\rangle_A) \\
& =\tau (\langle\zeta,\zeta\rangle_A)< \infty
\end{align*}
by the lower semicontinuity of $\tau$. Therefore $Tr_\tau^\mathcal{X}$ is 
densely defined on $K_A(\mathcal{X})$. It is easy to see that 
$Tr_\tau^\mathcal{X}$ is lower semicontinuous. 
\end{proof}
\begin{rem}
Since a right Hilbert $A$-module $\mathcal{X}$ is a $K_A(\mathcal{X})$-$A$-equivalence 
bimodule, 
a similar computation in the proof above implies $Tr_{Tr_\tau^\mathcal{X}}^{\mathcal{X}^*}=\tau$. 
Therefore there exists a bijective 
correspondence between densely defined lower semicontinuous traces on $A$ and those on 
$K_A(\mathcal{X})$. 
\end{rem}
To simplify notation, we use the same letter $\tau$ for the induced trace 
$Tr_{\tau}^{\mathcal{X}_{A}}$ on $M(A)$. 
We denote by $\hat{\tau}$ the induced trace $Tr_{\tau}^{H_A}$ on $M(A\otimes\mathbb{K})$. 

\section{Multiplicative maps of the Picard groups to $\mathbb{R}_{+}^\times$}
Let $A$ be a simple $\sigma$-unital $C^*$-algebra with unique (up to scalar multiple) 
densely defined lower semicontinuous trace $\tau_A$. 
Define a map $\hat{T}_{\tau_A}$ of $\mathcal{H}(A)$ to $[0,\infty]$ by 
\[\hat{T}_{\tau_A} ([\mathcal{X}]):= Tr_{\tau_A}^\mathcal{X}(1_{L_A(\mathcal{X})}).
\] 
We see that 
$\hat{T}_{\tau_A} ([\mathcal{X}])=\sum_{i=1}^\infty \tau_{A} (\langle \xi_i,\xi_i\rangle_A)$ 
where $\{\xi_i\}_{i=1}^\infty$ is a countable basis of $\mathcal{X}$ and 
it does not depend on the choice of basis (see Proposition \ref{pro:ind}). 
It is easily seen that $\hat{T}_{\tau_A}$ is well-defined. 
We shall compute $\hat{T}_{\tau_A} ([\overline{hH_A}])$ where $h$ is a 
positive element in $A\otimes \mathbb{K}$. 
Let $d_{\tau_A} (h)=\lim_{n\rightarrow \infty}\hat{\tau_A} (h^{\frac{1}{n}})$ for 
$h\in (A\otimes \mathbb{K})_{+}$. 
Then $d_{\tau_A}$ is a dimension function. (See, for example, \cite{Bla}, \cite{BH} and 
\cite{Cu}.) 
\begin{pro}\label{pro:value}
Let $A$ be a simple $\sigma$-unital $C^*$-algebra with unique (up to scalar multiple) 
densely defined lower semicontinuous trace $\tau_A$ and 
$h$ a positive element in $A\otimes \mathbb{K}$. 
Then $\hat{T}_{\tau_A} ([\overline{hH_A}])=d_{\tau_A}(h)$. 
\end{pro}
\begin{proof}
We may assume that $\| h\|\leq 1$. 
Then $\{h^{\frac{1}{n}}\}_{n\in\mathbb{N}}$ is an increasing approximate unit 
for $K_A(\overline{hH_A})$ and $\lim_{n\to\infty}h^{\frac{1}{n}}\xi =\xi$ 
for any $\xi\in\overline{hH_A}$. 
Let $\{\xi_i\}_{i\in\mathbb{N}}$ be a basis of $\overline{hH_A}$ and 
$\{\eta_j\}_{j\in\mathbb{N}}$ a basis of $H_A$. 
By the lower semicontinuity of $\tau_A$ and 
$\langle \xi_i,h^{\frac{1}{n}}\xi_i\rangle_A\leq 
\langle \xi_i,h^{\frac{1}{n+1}}\xi_i\rangle_A$, we have 
\begin{align*}
\lim_{n\rightarrow\infty}\sum_{i=1}^\infty\tau_{A}(\langle \xi_i,h^{\frac{1}{n}}\xi_i\rangle_A)
& =\sum_{i=1}^\infty\lim_{n\rightarrow\infty}\tau_{A}(\langle \xi_i,h^{\frac{1}{n}}\xi_i\rangle_A) \\
& =\sum_{i=1}^\infty\tau_{A}(\langle \xi_i, \xi_i\rangle_A)
=\hat{T}_{\tau_A} ([\overline{hH_A}]).
\end{align*}
In a similar way, we have 
\begin{align*}
\sum_{i=1}^\infty\tau_{A}(\langle \xi_i,h^{\frac{1}{n}}\xi_i\rangle_A)
& =\sum_{i=1}^\infty\tau_{A}(\langle \xi_i,h^{\frac{1}{2n}}
\sum_{j=1}^\infty\eta_j\langle\eta_j,h^{\frac{1}{2n}}\xi_i\rangle_A\rangle_A) \\
& =\sum_{i=1}^\infty\sum_{j=1}^\infty\tau_A
(\langle\xi_i,h^{\frac{1}{2n}}\eta_j\rangle_A\langle h^{\frac{1}{2n}}\eta_j,\xi_i\rangle_A) \\
& =\sum_{i=1}^\infty\sum_{j=1}^\infty\tau_A
(\langle h^{\frac{1}{2n}}\eta_j,\xi_i\rangle_A\langle\xi_i,h^{\frac{1}{2n}}\eta_j\rangle_A) \\
& =\sum_{j=1}^\infty\sum_{i=1}^\infty\tau_A
(\langle h^{\frac{1}{2n}}\eta_j, \xi_i\langle\xi_i,h^{\frac{1}{2n}} \eta_j\rangle_A\rangle_A) \\
& =\sum_{j=1}^\infty \tau_A(\langle\eta_j, h^{\frac{1}{n}}\eta_j\rangle_A )
=\hat{\tau_A}(h^{\frac{1}{n}}).
\end{align*}
Therefore $\hat{T}_{\tau_A} ([\overline{hH_A}])=d_{\tau_A}(h)$.
\end{proof}
\begin{rem}
Let $p$ be a projection in $M(A\otimes\mathbb{K})$. 
Then it is easy to see that $\hat{T}_{\tau_A} ([pH_A])=\hat{\tau_A}(p)$. 
\end{rem}
The following proposition is a generalization of Proposition 2.1 in \cite{NW2}. 
\begin{pro}\label{pro:multiplicative}
Let $A$ and $B$ be simple $\sigma$-unital $C^*$-algebras with unique (up to scalar multiple) 
densely defined lower semicontinuous traces $\tau_A$ and $\tau_B$ respectively. 
Assume that $\tau_A (1_{M(A)})=1$, that is, $\tau_A$ is a normalized trace. 
Then for every right Hilbert $A$-module $\mathcal{X}$ and 
every $A$-$B$-equivalence bimodule $\mathcal{F}$, 
\[\hat{T}_{\tau_B}([\mathcal{X}\otimes\mathcal{F}])=\hat{T}_{\tau_A}([\mathcal{X}])
\hat{T}_{\tau_B}([\mathcal{F}]).\]
\end{pro}
\begin{proof}
Let $\{\xi_i\}_{i\in\mathbb{N}}$ be a countable basis of $\mathcal{X}$ and 
$\{\eta_j\}_{j\in\mathbb{N}}$ a countable basis of $\mathcal{F}$ 
as a right Hilbert $B$-module. 
Then $\{\xi_i\otimes\eta_j \}_{i,j\in\mathbb{N}}$ is a countable basis of 
$\mathcal{X}\otimes\mathcal{F}$ as a right Hilbert $A$-module. 
By $\tau_B (\langle\xi_i\otimes\eta_j,\xi_i\otimes\eta_j\rangle_B)\geq 0$, 
we have  
\begin{align*}
\hat{T}_{\tau_B}([\mathcal{X}\otimes\mathcal{F}]) 
& =\sum_{i,j=1}^\infty\tau_B (\langle\xi_i\otimes\eta_j,\xi_i\otimes\eta_j\rangle_B) 
=\sum_{i=1}^\infty\sum_{j=1}^\infty \tau_B 
(\langle\eta_j, \langle\xi_i,\xi_i\rangle_A\eta_j\rangle_B ) \\
& =\sum_{i=1}^\infty Tr_{\tau_B}^{\mathcal{F}} (\langle\xi_i,\xi_i\rangle_A).
\end{align*} 
Since $Tr_{\tau_B}^{\mathcal{F}}$ is a densely defined lower semicontinuous trace on $A$, 
there exists $\lambda\in\mathbb{R}_{+}^\times$ such that 
$Tr_{\tau_B}^{\mathcal{F}} =\lambda \tau_A$. 
The assumption $\tau_A(1_{M(A)})=1$ implies $\lambda =Tr_{\tau_B}^{\mathcal{F}}(1_{M(A)})$. 
Therefore 
$$
\sum_{i=1}^\infty Tr_{\tau_B}^{\mathcal{F}} (\langle\xi_i,\xi_i\rangle_A)
=\sum_{i=1}^\infty Tr_{\tau_B}^{\mathcal{F}}(1_{M(A)})\tau_A(\langle\xi_i,\xi_i\rangle_A) 
=\hat{T}_{\tau_A}([\mathcal{X}])\hat{T}_{\tau_B}([\mathcal{F}]).
$$
\end{proof}
We shall consider the multiplicative map of the Picard group to $\mathbb{R}_{+}^\times$. 
\begin{pro}\label{pro:key}
Let $A$ be a simple $\sigma$-unital $C^*$-algebra with unique (up to scalar multiple) 
densely defined lower semicontinuous trace $\tau_A$. 
Assume that $\mathcal{X}$ is a nonzero right Hilbert $A$-module such that 
$\hat{T}_{\tau_A}([\mathcal{X}])<\infty$. 
Define a map $T_{\mathcal{X}}$ of $\mathrm{Pic}(K_A(\mathcal{X}))$ to 
$\mathbb{R}_{+}^\times$ by 
\[T_{\mathcal{X}}([\mathcal{E}]):=\frac{1}{\hat{T}_{\tau_A}([\mathcal{X}])}
\hat{T}_{\tau_A}([\mathcal{E}\otimes\mathcal{X}])\] 
for $[\mathcal{E}]\in\mathrm{Pic}(K_A(\mathcal{X}))$. 
Then $T_{\mathcal{X}}$ is well-defined and independent on the choice of trace. 
Moreover $T_{\mathcal{X}}$ is a multiplicative map. 
\end{pro}
\begin{proof}
If $K_A(\mathcal{X})$-$K_A(\mathcal{X})$ equivalence bimodule $\mathcal{E}^\prime$ 
is isomorphic to $\mathcal{E}$, then $\mathcal{E}^\prime\otimes\mathcal{X}$ is 
isomorphic to $\mathcal{E}\otimes\mathcal{X}$. 
Hence $T_{\mathcal{X}}([\mathcal{E}^\prime])=T_{\mathcal{X}}([\mathcal{E}])$. 
A similar computation in the proof of Proposition \ref{pro:multiplicative} shows 
$$T_{\mathcal{X}}([\mathcal{E}])=\frac{1}{\hat{T}_{\tau_A}([\mathcal{X}])}\hat{T}_{Tr_{\tau_A}^\mathcal{X}}([\mathcal{E}])
=\frac{1}{\hat{T}_{\tau_A}([\mathcal{X}])}Tr_{Tr_{\tau_A}^\mathcal{X}}^\mathcal{E}(1_{L_A(\mathcal{E})}).$$ 
Since $\mathcal{E}$ is a $K_A(\mathcal{X})$-$K_A(\mathcal{X})$-equivalence bimodule, 
$K_A(\mathcal{E})$ is isomorphic to $K_A(\mathcal{X})$. 
The uniqueness of the trace on $K_A(\mathcal{X})$ implies 
$$Tr_{Tr_{\tau_A}^\mathcal{X}}^\mathcal{E}(1_{L_A(\mathcal{E})})
=\lambda Tr_{\tau_A}^\mathcal{X} (1_{L_A(\mathcal{X})})=\lambda 
\hat{T}_{\tau_A}([\mathcal{X}])<\infty$$ 
for some $\lambda\in \mathbb{R}_{+}^\times$. Therefore $T_{\mathcal{X}}$ is well-defined.  
Define $\tau^\prime :=\frac{Tr_{\tau_A}^\mathcal{X}}{\hat{T}_{\tau_A}([\mathcal{X}])}$.
Then $\tau^\prime$ is a normalized trace on $K_A(\mathcal{X})$. 
By proposition \ref{pro:multiplicative}, 
\begin{align*}
T_{\mathcal{X}}([\mathcal{E}][\mathcal{E}^\prime ]) 
& =\frac{1}{\hat{T}_{\tau_A}([\mathcal{X}])}\hat{T}_{\tau_A}([\mathcal{E}
\otimes\mathcal{E}^\prime\otimes\mathcal{X}]) =
\frac{1}{\hat{T}_{\tau_A}([\mathcal{X}])}\hat{T}_{Tr_{\tau_A}^\mathcal{X}}
([\mathcal{E}\otimes\mathcal{E}^\prime ]) \\
& =\frac{1}{\hat{T}_{\tau_A}([\mathcal{X}])}\hat{T}_{\tau^\prime} ([\mathcal{E}])
\hat{T}_{Tr_{\tau_A}^\mathcal{X}}([\mathcal{E}^\prime ]) \\
& =T_{\mathcal{X}}([\mathcal{E}])T_{\mathcal{X}}([ \mathcal{E}^\prime ]). 
\end{align*}
\end{proof}
\section{Fundamental groups}
Let $A$ be a simple $\sigma$-unital $C^*$-algebra with unique (up to scalar multiple) 
densely defined lower semicontinuous trace $\tau$, and let $h_0$ be a nonzero 
positive element in $A\otimes\mathbb{K}$ with $d_{\tau}(h_0)<\infty$. 
Put 
\[\mathcal{F}_{h_0}(A) :=\left\{
d_\tau (h)/d_\tau (h_0) \in \mathbb{R}^{\times}_{+}\ | \
\begin{array}{lcr}
 h \text{ is a positive element in } A\otimes\mathbb{K} \text{ such that } \\
 \overline{h(A\otimes\mathbb{K})h} \cong \overline{h_0(A\otimes\mathbb{K})h_0} 
\end{array} 
\right\}\]

\begin{lem}\label{lem:group}
Let $A$ be a simple $\sigma$-unital $C^*$-algebra with unique (up to scalar multiple) 
densely defined lower semicontinuous trace $\tau$ and $h_0$ a nonzero positive element in 
$A\otimes \mathbb{K}$ such that $d_\tau (h_0)<\infty$. 
Then $\mathcal{F}_{h_0}(A)$ is a multiplicative subgroup of $\mathbb{R}_+^\times$. 
\end{lem}
\begin{proof}
Put $\mathcal{X}=\overline{h_0H_A}$. It is enough to show that 
$\mathcal{F}_{h_0}(A)= Im(T_{\mathcal{X}})$. 
Let $\mathcal{E}$ be a $K_A(\mathcal{X})$-$K_A(\mathcal{X})$-equivalence bimodule. 
Then there exists a positive element 
$h\in A\otimes\mathbb{K}$ such that $\mathcal{E}\otimes\mathcal{X}$ 
is isomorphic to $\overline{hH_A}$ as 
a right Hilbert $A$-module with an isomorphism of $K_A(\mathcal{X})$ to 
$\overline{h(A\otimes\mathbb{K})h}$ by Proposition \ref{pro:module}. 
Since $K_A(\mathcal{X})$ is isomorphic to $\overline{h_0(A\otimes\mathbb{K})h_0}$ and we have 
$T_{\mathcal{X}}([\mathcal{E}])=d_\tau (h)/d_\tau (h_0)$ by Proposition \ref{pro:value}, 
$Im(T_{\mathcal{X}})\subset \mathcal{F}_{h_0}(A)$. Conversely let $h$ be a positive 
element in $A\otimes\mathbb{K}$ such that $\overline{h(A\otimes\mathbb{K})h}$ is 
isomorphic to $\overline{h_0(A\otimes\mathbb{K})h_0}$.  
Since $A$ is simple and $\overline{h(A\otimes\mathbb{K})h}$ is isomorphic to 
$K_A(\mathcal{X})$, $\mathcal{E}:=\overline{hH_A}\otimes\mathcal{X}^*$ is a 
$K_A(\mathcal{X})$-$K_A(\mathcal{X})$-equivalence bimodule. 
By Proposition \ref{pro:value}, 
$T_{\mathcal{X}}([\mathcal{E}])=\frac{1}{\hat{T}_{\tau_A}([\mathcal{X}])}
\hat{T}_{\tau_A}([\overline{hH_A}])
=d_\tau (h)/d_\tau (h_0)$. 
Therefore $\mathcal{F}_{h_0}(A)\subset Im(T_{\mathcal{X}})$.

\end{proof}

\begin{lem}\label{lem:equivalence}
Let $A$ be a simple $\sigma$-unital $C^*$-algebra with unique (up to scalar multiple) 
densely defined lower semicontinuous trace $\tau$. 
Assume that $h_0$ and $h_1$ are nonzero positive elements in $A\otimes\mathbb{K}$ 
such that $d_{\tau}(h_0), d_{\tau}(h_1)<\infty$. 
Then $\mathcal{F}_{h_0}(A)=\mathcal{F}_{h_1}(A)$. 
\end{lem}
\begin{proof}
Let $\mathcal{F}:=\overline{h_0H_A}\otimes (\overline{h_1H_A})^*$. 
Then $\mathcal{F}$ is a $\overline{h_0(A\otimes\mathbb{K})h_0}$-
$\overline{h_1(A\otimes\mathbb{K})h_1}$-equivalence bimodule by the simplicity 
of $A$, and $\mathcal{F}$ induces an isomorphism $\Psi$ of 
$\mathrm{Pic}(\overline{h_0(A\otimes\mathbb{K})h_0})$ to 
$\mathrm{Pic}(\overline{h_1(A\otimes\mathbb{K})h_1})$ such that 
$\Psi ([\mathcal{E}])=[\mathcal{F}^*\otimes\mathcal{E}\otimes\mathcal{F}]$ 
for $[\mathcal{E}]\in\mathrm{Pic}(\overline{h_0(A\otimes\mathbb{K})h_0})$. 
By Proposition \ref{pro:multiplicative}, 
$T_{\overline{h_1H_A}}(\Psi ([\mathcal{E}]))=T_{\overline{h_0H_A}}([\mathcal{E}])$. 
Therefore $\mathcal{F}_{h_0}(A)=\mathcal{F}_{h_1}(A)$ by the proof of Lemma \ref{lem:group}. 
\end{proof}
Put 
$$
\mathcal{F}(A) :=\left\{d_\tau (h_1)/d_\tau (h_2) \in \mathbb{R}^{\times}_{+}\ | \
\begin{array}{lcr} 
 h_1 \text{ and } h_2 \text{ are nonzero positive elements in} \\ 
A\otimes\mathbb{K} \text{ such that } \\
\overline{h_1(A\otimes\mathbb{K})h_1} \cong 
\overline{h_2(A\otimes\mathbb{K})h_2}, d_\tau (h_2)<\infty 
\end{array} 
\right\}. 
$$
\begin{thm}\label{thm:main}
Let $A$ be a simple $\sigma$-unital $C^*$-algebra with unique (up to scalar multiple) 
densely defined lower semicontinuous trace $\tau$.
Then $\mathcal{F}(A)$ is a multiplicative subgroup of $\mathbb{R}^{\times}_{+}$. 
\end{thm}
\begin{proof}
Let $h_0$ be a nonzero positive element in $Ped(A\otimes\mathbb{K})$. 
By \cite{Ped2} (Proposition 5.6.2), 
$\overline{h_0(A\otimes\mathbb{K})h_0}$ is contained in $Ped(A\otimes\mathbb{K})$. 
 Since $\hat{\tau}$ is densely defined, $Ped(A\otimes\mathbb{K}) \subset 
\mathcal{M}_{\hat{\tau}}$. 
Therefore $\hat{\tau}$ is bounded on $\overline{h_0(A\otimes\mathbb{K})h_0}$ and hence 
$d_{\tau}(h_0)<\infty$. 
Lemma \ref{lem:equivalence} implies 
$\cup_{d_{\tau}(h)<\infty }\mathcal{F}_{h}(A)=\mathcal{F}_{h_0}(A)$. 
It is clear that $\mathcal{F}(A)=\cup_{d_{\tau}(h)<\infty}\mathcal{F}_{h}(A)$. 
Consequently $\mathcal{F}(A)$ is a multiplicative subgroup of $\mathbb{R}^{\times}_{+}$ 
by Lemma \ref{lem:group}.
\end{proof}

\begin{Def}\label{def:funda}
Let $A$ be a simple $\sigma$-unital $C^*$-algebra with unique (up to scalar multiple) 
densely defined lower semicontinuous trace $\tau$.
We call 
$\mathcal{F}(A)$ the fundamental group of $A$, which is a 
multiplicative subgroup of $\mathbb{R}^{\times}_{+}$.
\end{Def}
\begin{rem}\label{rem:multiplier}
It is easy to see that 
$\mathcal{F}(A)$ is equal to the set 
$$
\left\{\hat{\tau} (p)/\hat{\tau}(q) \in \mathbb{R}^{\times}_{+}\ | \
\begin{array}{lcr} 
 p \text{ and } q \text{ are nonzero projections in} \\ 
M(A\otimes\mathbb{K}) \text{ such that } \\
p(A\otimes\mathbb{K})p \cong 
q(A\otimes\mathbb{K})q, \hat{\tau} (q)<\infty 
\end{array} 
\right\}. 
$$

\end{rem}
\begin{rem}
If a unique densely defined lower semicontinuous trace $\tau$ is a normalized 
trace, then $\mathcal{F}(A)$ is equal to the set 
$$
\{d_\tau (h) \in \mathbb{R}^{\times}_{+}\ | \ 
h \text{ is a positive element in }  A\otimes\mathbb{K} \text{ such that } \\
A \cong \overline{h(A\otimes\mathbb{K})h}
\}. 
$$
Note that there exists a $\sigma$-unital simple $C^*$-algebra with a unique normalized trace 
$\tau$, which has a densely defined lower semicontinuous trace 
that is not a scalar multiple of $\tau$. 
For example, let $A$ be an AF-algebra such that 
$K_0(A)=\mathbb{Z}[\frac{1}{2}]\oplus\mathbb{Z}[\frac{1}{2}]$, 
$K_0(A)_{+}=\{(q,r)\in K_0(A):q>0,r>0\}\cup\{(0,0)\}$ and 
$\Sigma (A)=\{(q,r)\in K_0(A)_{+}:q>0,0<r<1\}\cup\{(0,0)\}$. 
Then $A$ is such a $C^*$-algebra. 
\end{rem}
The following corollary is shown by a similar argument of Lemma \ref{lem:equivalence}. 
\begin{cor}\label{cor:equivalence}
Let $A$ and $B$ be simple $\sigma$-unital $C^*$-algebras with unique (up to scalar multiple) 
densely defined lower semicontinuous traces. 
If $A$ is Morita equivalent to $B$, then $\mathcal{F}(A)=\mathcal{F}(B)$.
\end{cor}

We shall show that if $A$ is unital, then Definition \ref{def:funda} coincides 
with previous definition in \cite{NW} and \cite{NW2}. 
\begin{pro}\label{pro:unit}
Let $A$ be a unital simple $C^*$-algebra with a unique normalized trace 
$\tau$. Then 
$$
\mathcal{F}(A) =\{ \tau\otimes Tr(p) \in \mathbb{R}^{\times}_{+}\ | \ 
 p \text{ is a projection in } M_n(A) \text{ such that } pM_n(A)p  \cong A \} 
$$
where $Tr$ is the usual unnormalized trace on $M_n(\mathbb{C})$. 
\end{pro}
\begin{proof}
Let $\mathcal{X}_A$ be a right Hilbert $A$-module $A$ with the obvious right $A$-action and 
$\langle a ,b\rangle_A = a^*b$ for $a,b \in A$.  
Since $\tau$ is a normalized trace, $\hat{T}_{\tau}([\mathcal{X_A}])=1$. 
By the proof of Lemma \ref{lem:group} and Lemma \ref{lem:equivalence}, 
$\mathcal{F}(A)=\mathcal{F}_{1\otimes e_{11}}(A)= \hat{T}_{\tau}(\mathrm{Pic}(A))$ where 
$e_{11}$ is a rank one projection in $\mathbb{K}$. 
A similar argument in \cite{NW} (Theorem 3.1) shows 
$\hat{T}_{\tau}(\mathrm{Pic}(A))=\{ \tau\otimes Tr(p) \in \mathbb{R}^{\times}_{+}\ | \ 
 p \text{ is a projection in } M_n(A) \text{ such that } pM_n(A)p  \cong A \}$ 
because every $A$-$A$-equivalence bimodule has a finite basis. 
\end{proof}
We showed that K-theoretical obstruction enables us to compute 
fundamental groups easily in the case $A$ is unital \cite{NW}. 
Therefore if $A\otimes\mathbb{K}$ has a nonzero projection, 
we can compute fundamental groups easily by K-theoretical obstruction. 
We denote by 
$\tau_*$ the map $K_0(A)\rightarrow \mathbb{R}$ induced by a trace $\tau$ 
on $A$. 

\begin{Def}
Let $E$ be an additive subgroup of $\mathbb{R}$ containing $\mathbb{Z}$. 
Then an {\it inner multiplier group} $IM(E)$ of $E$ is defined by 
$$
IM(E) = \{t \in {\mathbb R}^{\times} \ | t \in E, t^{-1}  \in E, \text{ and } 
        tE = E \}.
$$
Then  $IM(E)$ is a multiplicative subgroup of $\mathbb{R}^{\times}$. 
We call  $IM_+(E) := IM(E) \cap  \mathbb{R}_+$ the 
{\it positive inner multiplier group} of $E$, which is a multiplicative subgroup of $\mathbb{R}^{\times}_+$.   
\end{Def}

\begin{cor}
Let $A$ be a separable simple $C^*$-algebra with unique (up to scalar multiple) 
densely defined lower semicontinuous trace $\tau$. Assume that $A\otimes\mathbb{K}$ 
has a nonzero projection. Then $\mathcal{F}(A)$ is countable. Moreover 
$\tau_*(K_0(A))$ is a $\mathbb{Z}[\mathcal{F}(A)]$-module and 
$\mathcal{F}(A) \subset IM_+(\tau_*(K_0(A)))$. 
\end{cor}
\begin{proof}
Let $p$ be a nonzero projection in $A\otimes\mathbb{K}$. 
Corollary \ref{cor:equivalence} implies 
$\mathcal{F}(A)=\mathcal{F}(p(A\otimes\mathbb{K})p)$. 
Since $p(A\otimes\mathbb{K})p$ is a separable unital $C^*$-algebra, 
\cite{NW} (Proposition 3.7) and Proposition \ref{pro:unit} prove 
the corollary. 
\end{proof}
\begin{ex}
Let $\mathbb{F}_n$ be a non-abelian free group with $n\geq 2$ generators. Then $C_r^*(\mathbb{F}_n)$ 
is a unital simple $C^*$-algebra with a unique normalized trace. Since $K_0(C_r^*(\mathbb{F}_n))
\cong \mathbb{Z}$, $\mathcal{F}(C_r^*(\mathbb{F}_n))=\{1\}$.
This implies that for positive elements $h_1, h_2\in C_r^*(\mathbb{F}_n)$ if 
$\overline{h_1C_r^*(\mathbb{F}_n)h_1}$ is isomorphic to $\overline{h_2C_r^*(\mathbb{F}_n)h_2}$, 
then $d_{\tau}(h_1)=d_{\tau}(h_2)$. 
\end{ex}
\begin{ex}
Let $p$ be a prime number. Consider a tensor product algebra of a UHF algebra and the compact 
operators 
$A=M_{p^\infty}\otimes\mathbb{K}$. Then $\mathcal{F}(A)=\{p^n:n\in \mathbb{Z}\}$.
\end{ex}
\begin{rem}
Any countable subgroup of $\mathbb{R}^\times_{+}$ can be realized as 
the  fundamental group $\mathcal{F}(A)$  of a separable 
simple unital $C^*$-algebra $A$ with a unique trace. (See \cite{NW2}.)
\end{rem}

We show that there exist separable simple stably projectionless 
$C^*$-algebras such that their fundamental groups are equal to $\mathbb{R}_+^\times$. 
This is a contrast to the unital case. 
Recall the building blocks that are considered by Razak \cite{Raz} and Tsang \cite{Tsa1}. 
These algebras are subhomogeneous algebras obtained by generalized mapping torus 
construction as in \cite{Ell} and \cite{EV}. 
For a pair of natural numbers $(n,m)$ with $n$ dividing $m$ $(m>n)$, 
let $\rho_0$ and $\rho_1$ be homomorphisms from $M_n(\mathbb{C})$ to $M_m(\mathbb{C})$, 
which having multiplicities $\frac{m}{n}-1$ and $\frac{m}{n}$ respectively. 
Define 
\[A(n,m)= \{f\in M_m(C([0,1])): f(0)=\rho_0 (c), f(1)=\rho_1 (c), c\in M_n(\mathbb{C}) \}.\] 
Note that we may assume that the homomorphism $\rho_0$ maps $M_n(\mathbb{C})$ into 
diagonal block matrices in $M_m(\mathbb{C})$ with $\frac{m}{n}-1$ identical blocks and 
one zero block, on the other hand, the homomorphism $\rho_1$ yields matrices with 
$\frac{m}{n}$ identical blocks. 
The building block $A(n,m)$ has the following properties. 
(See, for example, \cite{Ped2} and \cite{Raz}.) 
\begin{pro}\label{pro:building}
We have the following. \ \\
(i) Every primitive ideal of $A(n,m)$ is the kernel of some point evaluation. 
Therefore the primitive ideal space of $A(n,m)$ is homeomorphic to $\mathbb{T}$. \ \\
(ii) The Pedersen ideal of $A(n,m)$ is $A(n,m)$. Therefore every 
densely defined lower semicontinuous trace on $A(n,m)$ is bounded. \ \\
(iii) For any bounded trace $\tau$ on $A(n,m)$, there exists a measure $\mu$ on $\mathbb{T}$ 
such that $\tau (f)=\int_{\mathbb{T}}(\frac{m-n}{m})^{t}Tr (f(t)) d\mu (t)$ for any 
$f\in A(n,m)$. 
\end{pro} 
Fix an irrational $\theta\in [0,1]\backslash \mathbb{Q}$. 
For any $n\in\mathbb{N}$, define an injective homomorphism $\phi_{n}$ of 
$A(3^n, 2\cdot 3^n)$ to $A(3^{n+1}, 2\cdot 3^{n+1})$ by 
$$ (\phi_{n} (f))(t) =  \left\{\begin{array}{lcr} 
  u_t\left ( \begin{array}{ccc} 
    
        f(t)  &     0         &  0  \\ 
          0   & f(t+\theta ) &   0 \\
          0   &     0        &   0
\end{array} \right )  u_t^*  & 0 \leq t \leq 1-\theta  \\  
w_t  \left ( \begin{array}{ccc} 
    
        f(t) &       0       &  0   \\ 
          0   & f(t+\theta -1 ) &  0  \\
          0   &      0        & f(t+\theta -1 ) 
\end{array} \right )

w_t^*  & 1-\theta\leq t \leq 1  \\
 \end{array} \right.  
$$
where $u_t$ and $w_t$ are suitable continuous paths in $U(M_{2\cdot 3^{n+1}}(\mathbb{C}))$. 
We denote by $\phi_{n,m}$ a homomorphism $\phi_{m-1}\circ\cdot\cdot\cdot\phi_{n}$ 
from $A(3^{n}, 2\cdot 3^{n})$ to $A(3^{m}, 2\cdot 3^{m})$. 
Let $\mathcal{O}=\lim\limits_{\longrightarrow} (A(3^n, 2\cdot 3^n), \phi_{n,m})$. 
\begin{lem}\label{lem:projectionless}
With notation as above 
$\mathcal{O}=\lim\limits_{\longrightarrow} (A(3^n, 2\cdot 3^n), \phi_{n,m})$ is a separable simple stably 
projectionless $C^*$-algebra with unique (up to scalar multiple) densely defined lower semicontinuous 
unbounded trace. 
\end{lem}
\begin{proof} 
Let $J$ be a proper two-sided closed ideal of $\mathcal{O}$, and let 
$J_n=\phi^{-1}_{n,\infty}(J\cap \phi_{n,\infty}(A(3^n, 2\cdot 3^n)))$. 
Then $J_n$ is a two-sided closed ideal of $A(3^n, 2\cdot 3^n)$, and denote by $F_n$ 
the corresponding closed set in $\mathbb{T}$. (See Proposition \ref{pro:building}.) 
Since $\phi_{n,m}$ is injective, $J=\lim_{\rightarrow}(J_n,\phi_{n,m})$ and 
for $n$ sufficiently large, $J_n$ is a proper two-sided closed ideal of 
$A(3^n, 2\cdot 3^n)$, that is, $F_n$ is not empty. 
Put $\gamma ([t]) =[t+\theta ]$ for any $[t]\in\mathbb{T}$. 
For any natural number $k$, we see that 
$F_{n}=F_{n+k}\cup \gamma^{-1}(F_{n+k}) \cup\cdots\cup \gamma^{-k}(F_{n+k})$ 
by the construction of $\phi_{n,m}$ and $J_n=\phi^{-1}_{n,n+k}(J\cap \phi_{n,n+k}(A(3^n, 2\cdot 3^n)))$. 
A same argument in \cite{BK}(the last part of the proof of Proposition 1.3) shows that 
$\mathcal{O}$ is simple because $\gamma$ is minimal homeomorphism on $\mathbb{T}$. 

Define $\tau_n(f)=\frac{1}{(1+2^{\theta})^n}\int_{\mathbb{T}}(\frac{1}{2})^{t}Tr(f(t))d\mu (t)$ 
where $\mu$ is a normalized Haar measure on $\mathbb{T}$ and $Tr$ 
is the usual unnormalized trace on $M_{2\cdot 3^{n}}(\mathbb{C})$. Then 
$\tau_n=\tau_{n+1}\circ \phi_n$, and hence there exists a densely defined lower semicontinuous 
trace $\tau$ on $\mathcal{O}$. Note that $\tau$ is unbounded trace since 
$\| \tau_n \| =\frac{2\cdot 3^{n}}{(1+2^{\theta})^n}$. 

We shall show that the uniqueness of $\tau$. Let $\tau^{\prime}$ be a densely defined 
lower semicontinuous trace on $\mathcal{O}$. It is easy to see that 
$\tau^{\prime}|_{A(3^n, 2\cdot 3^n)}$ is densely defined lower semicontinuous trace 
on $A(3^n, 2\cdot 3^n)$. Proposition \ref{pro:building} implies that for any $n\in\mathbb{N}$ 
there exists a measure $\nu_{n}$ on $\mathbb{T}$ such that 
$\tau^{\prime}|_{A(3^n, 2\cdot 3^n)}(f)=\frac{1}{(1+2^{\theta})^n}
\int_{\mathbb{T}}(\frac{1}{2})^{t}Tr(f(t))d\nu_{n} (t)$. 
By a compatibility condition, we have 
$$ \int_{\mathbb{T}}(\frac{1}{2})^{t}Tr(f(t))d\nu_{n} (t)= 
\frac{1}{1+2^{\theta}}\int_{\mathbb{T}}(\frac{1}{2})^{t}(Tr(f(t))+g(t))d\nu_{n+1} (t)
$$
where $$g(t) =  \left\{\begin{array}{lcr}
            Tr (f(t+\theta ) ) & 0\leq t\leq 1-\theta \\
            2Tr (f(t+\theta -1 ) ) & 1-\theta\leq t\leq 1 \\
 \end{array} \right.$$ 
for any $f\in A(3^n, 2\cdot 3^n)$. Therefore for any $h\in C(\mathbb{T})$, we have 
$$ \int_{\mathbb{T}}h(t) d\nu_{n} (t)= 
\frac{1}{1+2^{\theta}}\int_{\mathbb{T}}h(t)+2^{\theta}h(t+\theta)d\nu_{n+1} (t).
$$
In a same way in \cite{KK1}(the last part of the proof of Theorem 2.4), 
we see that $\nu_{n}$ is Haar measure on $\mathbb{T}$ by this condition. 
Consequently there exists a positive number $\lambda$ such that 
$\tau^{\prime}=\lambda \tau$. 
\end{proof}
The following lemma is an immediate consequence of the classification theorem 
of Razak \cite{Raz} and Tsang \cite{Tsa1}(Theorem 3.1). 
\begin{lem}\label{lem:universal}
Let $A$ be an simple separable $AF$ algebra with unique (up to scalar multiple) 
densely defined lower semicontinuous trace. 
Then $A\otimes\mathcal{O}$ is isomorphic to $\mathcal{O}$. 
\end{lem}
\begin{thm}\label{thm:Main}
There exist a separable simple stably projectionless nuclear $C^*$-algebra 
and non-nuclear $C^*$-algebra with unique (up to scalar multiple) densely 
defined lower semicontinuous trace such that 
their fundamental groups are equal to $\mathbb{R}_+^\times$. 
\end{thm}
\begin{proof}
For any $\lambda\in\mathbb{R}_{+}^\times$, there exists a separable unital simple 
$AF$ algebra $A_\lambda$ with a unique trace such that $\lambda\in\mathcal{F}(A_{\lambda})$ by 
Corollary 3.16 in \cite{NW}. Lemma \ref{lem:universal} implies 
$\lambda\in\mathcal{F}(\mathcal{O})$. Therefore $\mathcal{F}(\mathcal{O})=\mathbb{R}_{+}^\times$. 
Let $\mathbb{F}_n$ be a non-abelian free group with $n\geq 2$ generators. 
Then $\mathcal{O}\otimes C_r^*(\mathbb{F}_n)$ is a separable stably projectionless 
non-nuclear $C^*$-algebra with unique (up to scalar multiple) densely 
defined lower semicontinuous trace such that 
$\mathcal{F}(\mathcal{O}\otimes C_r^*(\mathbb{F}_n))=\mathbb{R}_{+}^\times$. 
\end{proof}
\begin{rem}
Let $h$ be a nonzero positive element in the Pedersen ideal of $\mathcal{O}$. 
Then $\overline{h\mathcal{O}h}$ is a separable stably projectionless $C^*$-algebra 
with a unique normalized trace such that 
$\mathcal{F}(\overline{h\mathcal{O}h})=\mathbb{R}_{+}^\times$. 
\end{rem}
\begin{rem}
Recently, Jacelon \cite{J} construct a simple, nuclear, stably projectionless $C^*$-algebra 
$W$ with a unique normalized trace, which shares some of the important properties 
of the Cuntz algebra $\mathcal{O}_{2}$. This $C^*$-algebra is an inductive limit of 
building blocks $A(n,m)$. Hence $W\otimes\mathbb{K}$ is isomorphic to $\mathcal{O}$ by 
the classification theorem of Razak \cite{Raz}. 
Therefore Corollary \ref{cor:equivalence} and Theorem \ref{thm:Main} imply 
$\mathcal{F}(W)=\mathbb{R}_{+}^{\times}$. 
\end{rem}

Recall that 
the fundamental group of a $\mathrm{II}_1$-factor $M$ is equal to the set of trace-scaling constants for 
automorphisms of $M\otimes B(\mathcal{H})$. We have a similar fact as discussed by Kodaka in \cite{kod3}. 
We define the set of trace-scaling constants for automorphisms:  
$$
\mathfrak{S}(A)
:= \{ \lambda \in \mathbb{R}^{\times}_+ \ | \ 
\hat{\tau} \circ \alpha = \lambda \hat{\tau} \text{ for some  } 
 \alpha \in \mathrm{Aut} (A\otimes K(\mathcal{H})) \ \}. 
$$

\begin{pro}
Let $A$ be a simple $\sigma$-unital $C^*$-algebra with unique (up to scalar multiple) 
densely defined lower semicontinuous trace $\tau$.
Then $\mathcal{F}(A)=\mathfrak{S}(A)$. 
\end{pro}
\begin{proof}
There exists a nonzero projection $p$ in $M(A\otimes\mathbb{K})$ such that 
$\hat{\tau}(p)<\infty$ by a similar argument in the proof of Theorem \ref{thm:main}. 
Let $\lambda \in \mathfrak{S}(A)$, then there exists an automorphism of $A\otimes\mathbb{K}$ 
such that $\hat{\tau} \circ \alpha(x)=\lambda \hat{\tau}(x)$ for $x\in\mathcal{M}_{\hat{\tau}}$. 
There exists an automorphism $\tilde{\alpha}$ of $M(A\otimes\mathbb{K})$ such that 
$\tilde{\alpha} (x)=\alpha (x)$ for $x\in A\otimes\mathbb{K}$.  
It is clear that $p(A\otimes \mathbb{K})p$ is 
isomorphic to $\tilde{\alpha}(p)(A\otimes\mathbb{K})\tilde{\alpha}(p)$. We have that  
$\hat{\tau}(\tilde{\alpha}(p))/\hat{\tau}(p)=\lambda$. 
Therefore $\lambda\in\mathcal{F}(A)$ by Remark \ref{rem:multiplier}. 

Conversely, let $\lambda\in\mathcal{F}(A)$. 
There exist projections $p$ and $q$ in $M(A\otimes\mathbb{K})$ such that 
$p(A\otimes\mathbb{K})p$ is isomorphic to $q(A\otimes\mathbb{K})q$ and 
$\lambda =\hat{\tau}(p)/\hat{\tau}(q)$. 
We denote by $\phi$ an isomorphism of $p(A\otimes\mathbb{K})p$ 
to $q(A\otimes\mathbb{K})q$. 
Since $p$ and $q$ are full projections, there exist partial isometries 
$w_1$ and $w_2$ in $(A\otimes\mathbb{K})\otimes\mathbb{K}$ such that 
$w_1^*w_1=I\otimes I$, $w_1w_1^*=p\otimes I$, $w_2^*w_2=I\otimes I$ and 
$w_2w_2^*=q\otimes I$ by Brown \cite{B}. 
Let 
$\psi : A \otimes \mathbb{K} \otimes \mathbb{K}
\rightarrow A \otimes \mathbb{K}$ be an isomorphism which 
induces the identity on the $K_0$-group. 
Define $\alpha =\psi\circ (ad w_2^*) \circ \phi \circ (ad w_1) \circ \psi^{-1}$. 
Then $\hat{\tau}\circ \alpha=\lambda \hat{\tau}$. Therefore $\lambda\in  \mathfrak{S}(A)$.     
\end{proof}

\begin{ex}\label{ex:KK}
Let $\{\lambda_1,...,\lambda_n \}$ be nonzero positive numbers such that 
the closed additive subgroup of $\mathbb{R}$ generated by $\{\lambda_1,...,\lambda_n \}$ 
is $\mathbb{R}$ 
and $\mathcal{O}_{n}$ the Cuntz algebra generated by $n$ isometries 
$S_1,...,S_n$. 
There exists a one-parameter automorphism group 
$\alpha :\mathbb{R}\rightarrow \mathrm{Aut}(\mathcal{O}_n)$ 
given by $\alpha_t(S_j)=e^{it\lambda_j}S_j$. 
Define $A:=\mathcal{O}_n\rtimes_\alpha\mathbb{R}$. 
Then $A$ is a simple stable separable $C^*$-algebra with unique (up to scalar multiple) 
densely defined lower semicontinuous trace $\tau$ and $\mathfrak{S}(A)=\mathbb{R}_{+}^{\times}$ 
\cite{KK1} and \cite{KK2}. 
Therefore $\mathcal{F}(A)=\mathbb{R}_{+}^{\times}$ by the corollary above. 
\end{ex}

Finally we state a direct relation between the fundamental group 
of $C^*$-algebras and that of von Neumann algebras.   
\begin{pro}
Let $A$ be a $\sigma$-unital infinite-dimensional simple $C^*$-algebra with 
unique (up to scalar multiple) densely defined lower semicontinuous trace $\tau$. 
Assume that $\tau$ is a normalized trace. 
Consider the GNS representation $\pi_{\tau} : A \rightarrow B(H_{\tau})$ and the associated 
factor $\pi_{\tau}(A)''$ of type $\mathrm{II}_1$. 
Then $\mathcal{F} (A) \subset \mathcal{F} (\pi_{\tau}(A)'')  $. 
In particular, if $\mathcal{F} (\pi_{\tau}(A)'') = \{ 1 \}$, then 
$\mathcal{F}(A) = \{1 \}$. 
\end{pro}
\begin{proof}
Let $h$ be a positive element in $A\otimes\mathbb{K}$ such that 
$A$ is isomorphic to $\overline{h(A\otimes\mathbb{K})h}$. 
We denote by $\tilde{\tau}$ the restriction of $\hat{\tau}$ on 
$\overline{h(A\otimes\mathbb{K})h}$. 
By the uniqueness of trace,  $\pi_{\tau}(A)''$ is isomorphic to 
$\pi_{\tilde{\tau}}(\overline{h(A\otimes\mathbb{K})h})''$. 
Define $p:=\int_{0}^{\| h \|}dE_{t}$ where $\{E_{t}:0\leq t\leq \| h \| \}$ is the 
spectral projections of $\pi_{\hat{\tau}}(h)$. 
Then $d_{\tau}(h)=\hat{\tau}(p)$. A standard argument shows 
$p\pi_{\hat{\tau}}(A\otimes\mathbb{K})''p$ is isomorphic to 
$\pi_{\tilde{\tau}}(\overline{h(A\otimes\mathbb{K})h})''$. 
Therefore $\mathcal{F} (A) \subset \mathcal{F} (\pi_{\tau}(A)'')  $. 
\end{proof}

\end{document}